\documentclass{amsart}

\makeindex
\usepackage[all,cmtip]{xy}
\usepackage{amssymb}
\usepackage{amsmath}
\usepackage{amsthm}
\usepackage{amscd}
\usepackage{amsfonts}
\usepackage{amssymb}
\usepackage[pdftex]{graphicx}
\usepackage{multirow}

\usepackage[usenames,dvipsnames,svgnames,table]{xcolor}

\usepackage{epic}
\usepackage{graphicx}
\usepackage{pgf,tikz}

\usepackage[bookmarks=true,bookmarksnumbered=true,
 pdftitle={},
 pdfsubject={},
 pdfauthor={},
 pdfkeywords={TeX; dvipdfmx; hyperref; color;},
 colorlinks=true, linkcolor=Maroon, citecolor=JungleGreen]
 {hyperref}

\usetikzlibrary{arrows}

\newtheorem{theoremint}{Theorem}[section]
\newtheorem{questionsec}[theoremint]{Question}

\newtheorem{theorem}{Theorem}[subsection]
\newtheorem{lemma}[theorem]{Lemma}
\newtheorem{lemmaint}[theoremint]{Lemma}
\newtheorem{proposition}[theorem]{Proposition}
\newtheorem{conjecture}[theorem]{Conjecture}

\theoremstyle{corollary}
\newtheorem{corollary}[theorem]{Corollary}
\newtheorem{corollaryint}[theoremint]{Corollary}

\theoremstyle{definition} 
\newtheorem{definition}[theorem]{Definition}
\newtheorem{definition-lemma}[theorem]{Definition-Lemma}

\newtheorem{example}[theorem]{Example}
\newtheorem{exampleint}[theoremint]{Example}

\theoremstyle{remark}
\newtheorem{remark}[theorem]{Remark}

\numberwithin{equation}{section}

\newcommand{\C}{\mathbb{C}}
\newcommand{\R}{\mathbb{R}}

\newcommand{\Q}{\mathbb{Q}}
\newcommand{\mc}{\mathcal}
\def\P{\mathbb{P}}

\DeclareMathOperator{\B}{\textbf{B}}

\def\Proj{\operatorname{Proj}}

\def\Cox{\operatorname{Cox}}
\def\mult{\operatorname{mult}}

\def\Mov{\operatorname{Mov}}

\def\Num{\operatorname{N}}

\def\Exc{\operatorname{Exc}}
\def\NE{\operatorname{NE}}
\def\Sym{\operatorname{Sym}}

\input xy
\xyoption{all}

\title[Factorization of anticanonical maps of Fano type varieties]{Factorization of anticanonical maps \\of Fano type varieties}

\begin{document}

\author{Sung Rak Choi}
\address{Center for Geometry and Physics, Institute of Basic Science (IBS), Pohang, Korea}
\email{sungrakc@ibs.re.kr}

\author{DongSeon Hwang}
\address{Department of Mathematics, Ajou University, Suwon, Korea}
\email{dshwang@ajou.ac.kr}

\author{Jinhyung Park}
\address{Department of Mathematical Sciences, KAIST, Daejeon, Korea}
\curraddr{School of Mathematics, Korea Institute for Advanced Study, Seoul, Korea}
\email{parkjh13@kaist.ac.kr}

\thanks{S.R. Choi is partially supported by the IBS (CA1305-02).
D. Hwang is partially supported by Basic Science Research Program through the National Research Foundation of Korea (NRF) funded by the Ministry of Education (2011-0022904).
J. Park is partially supported by the National Research Foundation of Korea (NRF) grant funded by the Korea government (MSIP) (No. 2013042157).}


\begin{abstract}
The purpose of the present paper is to generalize Sakai's work on anticanonical models of rational surfaces to varieties of Fano type.
We first prove a characterization of Fano type varieties using the singularities of anticanonical models.
Secondly,  we study the decomposition of the anticanonical map using the $K_X$-minimal model program.
\end{abstract}

\maketitle
\tableofcontents

\section{Introduction}

Throughout the paper, we only consider normal projective varieties defined over the field $\C$ of complex numbers.
To understand the structure of the varieties of Fano type, which form an important class in birational geometry, we study their anticanonical models. In \cite{HP2}, Hwang and Park  characterize surfaces of Fano type via their anticanonical models, and they reduce the classification of log del Pezzo pairs to that of log del Pezzo surfaces. Their work is based on Sakai's results (\cite{S}) on the anticanonical models of rational surfaces. In this paper, we consider a higher dimensional  generalization of Sakai's work.

First, we characterize the varieties of Fano type via the singularities of anticanonical models.
This gives a generalization of \cite[Theorem 1.1]{HP2} (cf. \cite[Theorem 4.3]{S}). We will prove Theorem \ref{antimodel} in Section \ref{antimodelsec}.

\begin{theoremint}\label{antimodel}
Let $X$ be a $\Q$-Gorenstein normal projective variety such that $-K_X$ is big and $R(-K_X)$ is finitely generated. Then $X$ is of Fano type if and only if the anticanonical model $Y=\Proj R(-K_X)$ is a klt Fano variety.
\end{theoremint}

See Section \ref{prelimsec} for definitions of the varieties of Fano type and anticanonical models.
For any variety of Fano type, we can always take a $\Q$-factorialization, and hence, the $\Q$-Gorensteinness assumption is not restrictive.
Our proof of Theorem \ref{antimodel} heavily relies on the Zariski decomposition.
More precisely, for the `if' direction, we construct an explicit boundary divisor $\Delta$ on $X$ such that $(X, \Delta)$ is a log Fano pair
from the negative part of the Zariski decomposition of the anticanonical divisor as in \cite{HP2}.
Theorem \ref{antimodel} was also obtained by Cascini and Gongyo (see \cite[Theorem 1.1]{CG}) using different methods. For `only if' direction, they use $-K_X$-MMP, and for `if' direction,
they use terminalization to construct a boundary divisor instead of the Zariski decomposition.

The quickest way to construct the anticanonical model $Y$ of a $\Q$-factorial variety $X$ of Fano type is to run $-K_X$-MMP.
By running the $-K_X$-MMP, we obtain a birational map $X\dashrightarrow Y'$ such that $-K_{Y'}$ is nef. Since $Y'$ is a Mori dream space, the nef divisor $-K_{Y'}$ is semiample. Thus we get the morphism $\varphi_{|-mK_{Y'}|} : Y' \rightarrow Y=\Proj R(-K_X)$ to the anticanonical model $Y$.

We propose another way to construct a birational map to the anticanonical model, which generalizes the surface case obtained by Sakai (\cite[Proposition 4.2 and Theorem 4.3]{S}) into higher dimensional cases.

\begin{theoremint}\label{antimap}
Let $X$ be a $\Q$-Gorenstein normal projective variety of Fano type, and let $f: X \dashrightarrow Y=\Proj R(-K_X)$ be the anticanonical map. Then we can take $\Q$-factorial varieties $X_q$, $X'$, $X'_{mint}$, and $X'_{nrd}$ having the anticanonical model $Y$ with the following commutative diagram

$$
\xymatrix{
& X'_{mint} \ar@{-->}[r]^{r}  \ar[d]_{\varphi} \ar[rd]^{\widetilde{f}} & X'_{nrd} \ar[d]^{\pi}\\
X_q \ar@{-->}[r]^s \ar[d]_q   & X' \ar[r]^{f'} & Y \\
X \ar@{-->}@/_1pc/[rru]^f &  &
}
$$

\noindent where $q$ is a $\Q$-factorialization, $s$ is a finite sequence of log flops,
$f'$ and $\widetilde{f}$ are anticanonical morphisms, $\varphi$ and $\pi$ are minimal terminal resolutions, and $r$ is a redundant minimal model program to a non-redundant model.
\end{theoremint}

For the definitions of minimal terminal resolution $X_{mint}$ and  non-redundant model $X_{nrd}$, see Proposition \ref{minterres} and Definition \ref{def-redundant ray}, respectively.
We introduce the notion of redundant MMP in Subsection \ref{redmmp}.
Roughly speaking, the redundant MMP finds a `\emph{minimal}' variety
among the varieties having the anticanonical model $Y$.
One advantage of this factorization is that all the varieties $X$, $X_q$, $X'$, $X'_{mint}$, and $X'_{nrd}$ appearing in the above diagram have the same anticanonical model $Y$.
Furthermore, we will also show that all these varieties are birationally dominated by only finitely many terminalizations of $Y$ (Corollary \ref{finite}).

Let us recall the Sakai's construction in \cite{S} for the surface case which inspired our result.
Let $X$ be a normal projective rational surface with $-K_X$ big, and let $f: X \to Y$ be the anticanonical morphism. Then for the minimal resolution $\varphi: X_{mint} \to X$, the anticanonical morphism $\widetilde{f}: X_{mint} \to Y$ factors through the minimal resolution $\pi : Y_{mint} \to Y$.
The variety $Y_{mint}$ can also be obtained by contracting all the redundant curves on $X_{mint}$ (\cite[Definition 4.1]{S}).
Such curves are redundant in the sense that their contraction to points does not affect the anticanonical model.
Since $Y_{mint}$ has no redundant curves, we call such variety \emph{non-redundant} and denote it by $X_{nrd}$.
This factorization yields the following commutative diagram
$$
\xymatrix{
X_{mint} \ar[rd]^{\widetilde{f}} \ar[d]_{\varphi}\ar[r]^{r} & X_{nrd} \ar[d]^{\pi}\\
X\ar[r]_f & Y.
}
$$
An advantage of having this decomposition is that it reduces the study of the anticanonical map to that of the morphism between
the smooth varieties $X_{mint}$ and $X_{nrd}$ whose anticanonical models are $Y$.

Clearly, this construction cannot be generalized directly in higher dimensions.
First of all, in higher dimensions, the map $f$ to the anticanonical model $Y$ is not a morphism in general.
Furthermore, there are no minimal resolutions of singularities in higher dimensions.
To overcome these difficulties, after passing to a $\Q$-factorialization $X_q$, we take a small birational map $s:X_q \dashrightarrow X'$ so that there is a morphism $X'\to Y$ (Lemma \ref{zarsqmfano})
and we use \emph{minimal terminal resolution}
$\varphi:X'_{mint}\to X'$ (Proposition \ref{minterres}).
Sakai's contraction $r$ of redundant curves is generalized to the \emph{redundant MMP} $r: X_{mint}' \dashrightarrow X_{nrd}'$ (Section \ref{redmmpsec}).

Most results extend to the case where the anticanonical model has log canonical singularities (Subsections \ref{lcsingsec} and \ref{antimapcysec}).
The main difference is that if $Y$ contains singularities worse than klt, there are infinitely many $\Q$-factorial varieties $X$ with $-K_X$ big having $Y$ as the anticanonical model (Proposition \ref{no maxt}). 
In particular, contrary to the klt case, there are no varieties birationally dominating all other varieties 
having $Y$ as the anticanonical model.

The paper is organized as follows.
We start in Section \ref{prelimsec} by collecting and reviewing some basic results.
Section \ref{antimodelsec} is devoted to prove theorem \ref{antimodel}.
We develop the notion of the redundant MMP in Section \ref{redmmpsec}.
Section \ref{antimapsec} presents the proof of Theorem \ref{antimap}.
Finally, we compare log Calabi-Yau pairs and log weak Fano pairs in Section \ref{cywfsec}.

\section{Preliminaries}\label{prelimsec}

In this section, we collect basic notions and facts which will be used throughout the paper.

\subsection{Stable base loci}

Let $D$ be a $\Q$-divisor on a normal projective variety $X$.
The stable base locus of $D$ is denoted by $\B(D)$.

\begin{definition}
The \emph{augmented base locus} of $D$ is defined as the set
$$
\B_+(D):=\bigcap_{A: \text{ ample $\Q$-divisor}} \B(D-A).
$$
The \emph{restricted base locus} of $D$ is defined as the set
$$
\B_-(D):=\bigcup_{A: \text{ ample $\Q$-divisor}} \B(D+A).
$$
\end{definition}

It is well known that $\B_+(D)$ and $\B_-(D)$ depend only on the numerical class $[D]\in\Num^1(X)_\Q:=\Num^1(X)\otimes\Q$.
We also remark that for any $\Q$-divisor $D$, $\B_+(D)$ is always Zariski closed
whereas $\B_-(D)$ is \textit{a priori} only a countable union of Zariski closed sets.

\begin{definition}
A $\Q$-divisor $D$ is called \emph{movable} if $\B_-(D)$ contains no irreducible divisors.
The \emph{movable cone} $\Mov(X) \subseteq \Num^1(X)_\R:=\Num^1(X)\otimes\R$ is the closure of the cone spanned by the classes of movable divisors.
\end{definition}

\subsection{Log pairs}

We call $(X, \Delta)$ a \emph{log pair} if $X$ is a normal projective variety and $\Delta$ is an effective $\Q$-divisor on $X$ such that $K_X + \Delta$ is $\Q$-Cartier.
For a log resolution $f : Z \rightarrow X$ of $(X, \Delta)$, we have
$$
K_Z + f_{*}^{-1}\Delta = f^{*}(K_X + \Delta) + \sum a_i E_i,
$$
where each $E_i$ is an $f$-exceptional prime divisor and $a_i = a(E_i; X, \Delta)$ is the \emph{discrepancy} of $(X,\Delta)$ at $E_i$. A log pair $(X, \Delta)$ is called \emph{terminal} (resp. \emph{canonical, Kawamata log terminal} (\emph{klt} for short)) if $\lfloor \Delta \rfloor =0$ and every $a_i >0$ (resp. $a_i \geq 0$, $a_i >-1$). A log pair $(X, \Delta)$ is called \emph{divisorial log terminal} (\emph{dlt} for short) (resp. \emph{log canonical} (\emph{lc} for short)) if every $a_i > -1$ (resp. $a_i  \geq -1$) and every coefficient of $\Delta$ is at most $1$. If $(X, 0)$ is a terminal (resp. canonical, klt, lc) pair, then we say that $X$ has \emph{terminal} (resp. \emph{canonical}, \emph{klt}, \emph{lc}) \emph{singularities}. See \cite{KM} for more details on the singularities of log pairs.

\begin{definition}\label{def-FT,sFT}
Let $X$ be a normal projective variety.
\begin{enumerate}
 \item A klt pair $(X, \Delta)$ for some $\Q$-divisor $\Delta$ is called a \emph{klt Fano pair} (resp. \emph{klt weak Fano pair}) if $-(K_X + \Delta)$ is ample (resp. nef and big).
 \item For an lc pair $(X, \Delta)$, \emph{lc Fano pair, lc weak Fano pair} are defined similarly.
 \item A variety $X$ is called a \emph{variety of Fano type} (resp. \emph{weak Fano type}) if there exists a $\Q$-divisor $\Delta$ on $X$ such that $(X, \Delta)$ is a klt Fano pair (resp. klt weak Fano pair). Moreover, if $\Delta = 0$, the variety $X$ is called a \emph{klt Fano variety} (resp. \emph{klt weak Fano variety}).
\end{enumerate}
\end{definition}

\begin{remark}
By Kodaira lemma, a Fano type variety coincides with a weak Fano type variety (\cite[Lemma-Definition 2.6]{PS}).
It is known that Fano type varieties are Mori dream spaces (\cite{BCHM}, \cite{HK}).
\end{remark}

\begin{definition}
Let $X$ be a normal projective variety.
\begin{enumerate}
 \item A log pair $(X, \Delta)$ is called a \emph{klt} (resp. an \emph{lc}) \emph{Calabi-Yau pair} if $(X, \Delta)$ is klt (resp. lc) and $-(K_X + \Delta) \sim_{\Q} 0$.
 \item If $(X, \Delta)$ is an lc Calabi-Yau pair for some $\Q$-divisor $\Delta$, then $X$ is called a \emph{variety of Calabi-Yau type}. If $(X, \Delta)$ is a klt Calabi-Yau pair for some $\Q$-divisor $\Delta$, then $X$ is called a \emph{variety of klt Calabi-Yau type}.
\end{enumerate}
\end{definition}

In general, a variety of klt Calabi-Yau type need not be of Fano type (e.g., a smooth K3 surface).
The following is easy to verify.

\begin{lemma}\label{cy}
Let $X$ be a normal projective variety. Then $X$ is a variety of Fano type if and only if $X$ is a variety of klt Calabi-Yau type and $-K_X$ is big.
\end{lemma}

For a given klt pair $(X,\Delta)$,
it is well known that by \cite[Corollary 1.4.3]{BCHM},
there exists a small birational morphism (called a \emph{$\Q$-factorialization})
$q:X'\to X$ from a $\Q$-factorial variety $X'$.
Another consequence of \cite[Corollary 1.4.3]{BCHM} is that
there also exists a birational morphism (called a \emph{terminalization})
$t:Z\to X$ where $(Z, \Delta_Z)$ is a $\Q$-factorial terminal pair and $\Delta_Z$ is an effective divisor such that $K_Z + \Delta_Z = t^*(K_X + \Delta)$.

In the study of lc pairs, the following modification is often useful (\cite[Theorem 4.1]{F2}).
If $(X,\Delta)$ is an lc pair, then there exists a birational morphism (called a \emph{dlt blow-up})
$d:X_d\to X$ from a $\Q$-factorial variety $X_d$
such that
$a(E;X,\Delta)=-1$ for every $d$-exceptional prime divisor $E$ on $X_d$
and $(X_d,\Delta_d)$ is dlt where $\Delta_d$ is a divisor on $X_d$ defined by
$K_{X_d}+\Delta_d = d^*(K_X + \Delta)$.


For a given normal projective variety $X$, a \emph{minimal terminal resolution} $X_{mint}$
is defined by the following

\begin{proposition}[minimal terminal resolution, cf.{\cite[Theorem 1.33]{Ko-sing}}]\label{minterres}
Let $X$ be a normal projective variety.
Then there is a $\Q$-factorial normal projective variety $X_{mint}$ (not necessarily unique) having terminal singularities
with a birational morphism (called a \emph{minimal terminal resolution}) $f : X_{mint} \rightarrow X$ such that $K_{X_{mint}}$ is $f$-nef and $f$-big.
\end{proposition}


\subsection{Anticanonical models}

Here we define the anticanonical rings and anticanonical models.

\begin{definition}
Let $X$ be a $\Q$-Gorenstein projective variety. The \emph{anticanonical ring} of $X$ is defined as
$$
R(-K_X)  := \bigoplus_{m \geq 0} H^0(\mathcal{O}_X(-mm_0K_X)),
$$
where $m_0$ is the smallest positive integer such that $m_0K_X$ is a Cartier divisor. In the case where $R(-K_X)$ is finitely generated, the \emph{anticanonical model} is defined as $Y:= \Proj R(-K_X)$.
\end{definition}

The following example is remarked by Shokurov.

\begin{example}
The anticanonical model of a $\Q$-Gorenstein toric variety is a toric Fano variety. Using the toric minimal model program, we can explicitly describe birational maps and varieties in Theorem \ref{antimap} by combinatorial ways.
\end{example}

If $X$ is a $\Q$-Gorenstein variety of Fano type, then $R(-K_X)$ is finitely generated. However, when $X$ is a $\Q$-Gorenstein variety of Calabi-Yau type with $-K_X$ big, $R(-K_X)$ is not always finitely generated by the following example.

\begin{example}
Let $S$ be a K3 surface of Picard rank two such that $\Cox(S)$ is not finitely generated. Consider the smooth projective variety $X := \P(\mathcal{O}_S(A) \oplus \mathcal{O}_S(B) \oplus \mathcal{O}_S(-A-B))$, where the cone generated by $A$ and $B$ contains $\overline{\NE}(S)$. By \cite[Proposition 2.10]{CG}, $X$ is of Calabi-Yau type. We may assume that $2A + 2B + (-A-B)=A+B$ is big. Then the tautological line bundle $\mathcal{O}_X(1)$ is big, and hence, so is $\mathcal{O}_X(-K_X)=\mathcal{O}_X(2)$. Since $\Cox(S)$ is a direct summand of $\bigoplus_{m \geq 0} H^0(\mathcal{O}_X(m))$, it follows that $R(-K_X)$ is not finitely generated.
\end{example}

\subsection{Zariski decomposition}
Let $X$ be a normal projective variety, and let $D$ be a pseudoeffective $\Q$-divisor on $X$.

\begin{definition}\label{def-Zardecomp}
The decomposition $D = P+N$ is called the \emph{Zariski decomposition} (in the sense of Cutkosky-Kawamata-Moriwaki) if
\begin{enumerate}
 \item the \emph{positive part} $P$ is a nef $\Q$-divisor,
 \item the \emph{negative part} $N$ is an effective $\Q$-divisor, and
 \item the natural map $H^0(\mathcal{O}_X(mP)) \rightarrow H^0(\mathcal{O}_X(mD))$ is isomorphic for every sufficiently divisible integer $m>0$.
\end{enumerate}
If the positive part $P$ is semiample, we say the decomposition is \emph{good}.
\end{definition}

The Zariski decomposition is unique if it exists. Except in dimension $2$, Zariski decompositions may not exist even for the pull backs of divisors in general (see \cite{Na}).
If $f^*D$ admits the Zariski decomposition for some birational morphism $f : W \to X$, we say $D$ admits a \emph{birational Zariski decomposition}.

\begin{remark}\label{remk-Zariski decomp}
(1) We can check that if $D=P+N$ is the Zariski decomposition on $X$, then $f^{*}D=f^{*}P + f^{*}N$ is the Zariski decomposition on $Z$ for any birational morphism $f : Z \rightarrow X$.\\
(2) The Zariski decomposition in Definition \ref{def-Zardecomp} coincides with the Zariski decomposition in the sense of Fujita if it exists and $D$ is big (\cite[III. Remark 1.17 (3)]{Na}).
Thus if $D=P+N$ is the Zariski decomposition and $D = P'+N'$ is a decomposition with $P'$ nef and $N'$ effective, then $P \geq P'$.
\end{remark}

The following is immediate by \cite[III. Remark 1.17 (4)]{Na} (cf. \cite[Lemma 1.6]{HK}).

\begin{lemma}\label{bigzariski}
Let $X$ be a normal projective variety and $D$ a big $\Q$-Cartier divisor on $X$. Then the section ring $R(D)$ is finitely generated if and only if $D$ admits the good birational Zariski decomposition.
\end{lemma}

The following is also well known.

\begin{lemma}\label{zarviaanti}
Let $X$ be a normal projective variety and $D$ a big $\Q$-Cartier divisor on $X$ such that the section ring $R(D)$ is finitely generated.
Then the following are equivalent:
\begin{enumerate}
 \item $D$ admits the good Zariski decomposition $D= P+N$.
 \item There is a birational morphism $f : X \rightarrow Y:=\Proj R(D)$.
\end{enumerate}
If one of these equivalent conditions holds, then $P\sim_\Q f^*(H)$ for some ample divisor $H$ on $Y$
and $N$ is an $f$-exceptional effective divisor.
\end{lemma}


\begin{proof}
Assume first that $D=P+N$ is the good Zariski decomposition. Since $P$ is semiample, we obtain the morphism $f:=\varphi_{|mP|} : X \rightarrow Y=\Proj R(D)$ for sufficiently divisible $m>0$.
Conversely, assume that there is a birational morphism $f : X \rightarrow Y:=\Proj R(D)$. By \cite[Lemma 1.6]{HK}, there is an ample divisor $H$ on $Y$ such that $N=D - f^*H$ is effective. It is easy to check that $D=P+N$ with $P=f^*H$ is the good Zariski decomposition.
\end{proof}


\section{Singularities of anticanonical models}\label{antimodelsec}

Let $X$ be a smooth projective variety. If $K_X$ is big and the canonical ring $R(K_X)$ is finitely generated, then the canonical model $\Proj R(K_X)$ has canonical singularities.
However, if $-K_X$ is big and the anticanonical ring $R(-K_X)$ is finitely generated, the anticanonical model $\Proj R(-K_X)$ can be very singular in general.
Our goal in this section is to answer the following simple question.

\begin{questionsec}
When does the anticanonical model $\Proj R(-K_X)$ have klt or lc singularities?
\end{questionsec}

Although this question is also answered in \cite{CG}, we present a slightly different proof below.

\subsection{Klt singularities (Proof of Theorem \ref{antimodel})}
We prove Theorem \ref{antimodel}. Recall that $X$ is a $\Q$-Gorenstein variety such that $-K_X$ is big and $R(-K_X)$ is finitely generated.
By Lemma \ref{bigzariski}, there is a birational morphism $f : W \rightarrow X$ from a smooth variety $W$ such that $f^{*}(-K_X)=P+N$ is the good Zariski decomposition.
Thus we have
$$
H^0(\mathcal{O}_X(-mK_X))= H^0(\mathcal{O}_W (f^{*}\mathcal{O}_X(-mK_X)))=H^0(\mathcal{O}_W(mP))
$$
for any sufficiently divisible integer $m > 0$. In particular, we obtain the birational morphism $g:=\varphi_{|m_0P|} : W \rightarrow Y := \Proj R(-K_X)$ to the anticanonical model of $X$ for some integer $m_0 >0$
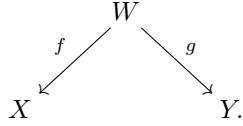
\[
\xymatrix{
& \ar[ld]_f W \ar[rd]^g & \\
X && Y .\\
}
\]
Possibly by taking further blow-ups, we may assume that $g$ is a log resolution of $(Y,0)$. Note that $P=g^{*}(-K_Y)$. We write $-K_W = f^{*}(-K_X)+F$ for some (not necessarily effective) $f$-exceptional divisor $F$. Then we have
$$
-K_W = P+N+F = g^{*}(-K_Y) + (N+F).
$$

Now suppose that $X$ is a $\Q$-Gorenstein variety of Fano type. To show that $Y$ has only klt singularities, we prove that the coefficients of $N+F$ are smaller than $1$.
By Lemma \ref{cy}, there is an effective $\Q$-divisor $\Delta$ such that $(X, \Delta)$ is a klt Calabi-Yau pair. Possibly by taking further blow-ups, we may assume that $f$ is a log resolution of $(X, \Delta)$. Since $P$ is semiample and big, there is an effective $\Q$-divisor $P' \sim_\Q P$ such that
$$
-K_W = f^{*}(-K_X-\Delta) + P'+N+F = g^{*}(-K_Y) + (N+F).
$$
Since $(X, \Delta)$ is a klt pair, the coefficients of $P'+N+F$ are less than $1$, and so are the coefficients of $N+F$. Thus $Y$ has klt singularities.

Conversely, suppose that the anticanonical model $Y$ is a klt Fano variety.
By Bertini Theorem, for sufficiently large and divisible $m_1>0$,
we can choose a general member $G$ in $|m_1 P|$ which is a smooth prime divisor.
Let $P_1 := \frac{1}{m_1} G$. Then $P_1+ N + F$ has a snc support possibly by taking further blow-ups of $W$.
Consider $-K_W=g^{*}(-K_{Y}) + N + F.$
Since $Y$ has klt singularities, every coefficient of $N+F$ is less than $1$.
Thus every coefficient of $P_1+N+F$ is less than $1$. Now, put $\Delta := f_{*}(P_1 + N+F)$.
Since $P_1+N+F\sim_\Q P+N+F=f^*(-K_X)+F$ and $F$ is $f$-exceptional, we have
$\Delta \sim_{\Q} -K_X$ and $\lfloor \Delta \rfloor =0$. We claim that $(X, \Delta)$ is a klt pair.
Recall that $-K_W = f^{*}(-K_X - \Delta) + P_1 + N + F$.
Since $P_1+N+F$ has a snc support, $f$ is a log resolution of $(X, \Delta)$. Furthermore, every coefficient of $P_1+N+F$ is less than $1$. Thus the claim follows, i.e., $(X, \Delta)$ is a klt Calabi-Yau pair. By Lemma \ref{cy}, we complete the proof of Theorem \ref{antimodel}. \hfill $\square$

\subsection{Lc singularities}\label{lcsingsec}
By the same argument as in Proof of Theorem \ref{antimodel}, we can also characterize a variety of lc Calabi-Yau type via the singularities on the anticanonical model (cf. \cite[Theorem 1.1]{HP2}, \cite[Corollary 3.6]{CG}). We leave the proof to the interested readers.

\begin{theorem}\label{lcsing}
Let $X$ be a $\Q$-Gorenstein variety such that $-K_X$ is big and $R(-K_X)$ is finitely generated. Then $X$ is of Calabi-Yau type if and only if the anticanonical model $Y=\Proj R(-K_X)$ is an lc Fano variety.
\end{theorem}

\begin{remark}
In Theorem \ref{antimodel} and Theorem \ref{lcsing}, the bigness of the anticanonical divisor is essential. To see this, let $X$ be an extremal rational elliptic surface $X_{22}$ in \cite[Theorem 4.1]{MP}. There is a singular fiber consisting of nine $(-2)$-curves $E_1, \ldots, E_9$ forming the dual graph $\widetilde{E}_8$. We can write
$$
-K_X = 2E_1 + 4E_2 + 6E_3 + 3E_4 + 5E_5 + 4E_6 + 3 E_7 + 2E_8 + E_9.
$$
Let $\pi : \widetilde{X} \rightarrow X$ be the blow-up at the intersection point $p$ of $E_8$ and $E_9$ with the exceptional divisor $E$.
It is easy to see that the Zariski decomposition
$-K_{\widetilde{X}} =P+N$ is given by $P=\frac{2}{3}\pi^{*}(-K_X)$ and
$$
N = \frac{1}{3}(2E_1 + 4E_2 + 6E_3 + 3E_4 + 5E_5 + 4E_6 + 3 E_7 + 2E_8 + E_9) +2E,
$$
where we use the same notations for the strict transforms. Since some coefficient of $N$ is larger than 1, it follows that $\widetilde{X}$ cannot be of Calabi-Yau type.
However, the anticanonical model of $\widetilde{X}$ is a smooth rational curve $\P^1$.
\end{remark}

As an application of Theorems \ref{antimodel} and \ref{lcsing}, we obtain the following, which can also be shown by a direct argument.

\begin{corollary}\label{cor-CY is closed}
Let $X$ be a $\Q$-Gorenstein variety such that $-K_X$ is big and $R(-K_X)$ is finitely generated.
For every small birational map $X \dashrightarrow X'$ to a $\Q$-Gorenstein variety $X'$, if $X$ is of Fano type (resp. of Calabi-Yau type), then so is $X'$.
\end{corollary}

\section{Redundant minimal model program}\label{redmmpsec}
In this section, we introduce the notion of redundant MMP, which generalizes the redundant contractions for surfaces (\cite{HP1}, \cite{HP2}, \cite{S}).

\subsection{Redundant MMP}\label{redmmp}
We begin with our motivation of the redundant MMP by reviewing Sakai's redundant contractions/blow-ups in dimension 2 (see \cite{HP1}, \cite{HP2}, \cite{S} for details). Let $S$ be a smooth projective rational surface such that $-K_S$ is big. Then we have the good Zariski decomposition $-K_S=P+N$.
A $(-1)$-curve $E$ on $S$ is called a \emph{redundant curve} if $E \cdot P=0$ (\cite[Definition 4.1]{S}).
Sakai in \cite{S} showed that such curves are redundant in the sense that if $\varphi : S\to S'$ is a birational morphism contracting some redundant curves, then the anticanonical models of $S$ and $S'$ remain unchanged. Note that the contraction morphism $\varphi : S\to S'$ can be seen as a partial MMP. This is the starting point of the redundant MMP that we develop below.

From now on, let $Z$ be a $\Q$-factorial klt variety such that $-K_Z$ is big. Any $K_Z$-negative extremal ray $R$ of $\overline{\NE}(Z)$ is contractible and there exists a contraction $\varphi_R:Z\to Z'$ associated to $R$ by the cone Theorem.

\begin{definition}\label{def-redundant ray}
The notations are as above.
\begin{enumerate}
\item $R$ is called a \emph{redundant extremal ray} if the exceptional locus $\Exc(\varphi_R)$ is contained in $\B_+(-K_Z)$. The contraction $\varphi_R$ associated to a redundant extremal ray $R$ is called the \emph{redundant contraction} of $R$.
\item If there are no redundant extremal rays in $\overline{\NE}(Z)$, then $Z$ is called a \emph{non-redundant model}.
\end{enumerate}
\end{definition}

We will see that if $-K_Z$ admits the good Zariski decomposition $-K_Z=P+N$, then an extremal ray $R$ is redundant if and only if $R \cdot P=0$ (see Proposition \ref{prop-redundant}).
Thus redundant extremal rays can be considered as a generalization of Sakai's redundant curves.

\begin{lemma}\label{lem-phi_R=bir}
Every redundant contraction $\varphi_R$ is birational.
That is, $\varphi_R$ is either of divisorial type or small type.
\end{lemma}
\begin{proof}
Suppose that $\varphi_R$ is not birational. Then we have $\Exc(\varphi_R)=Z$ so that $\B_+(-K_Z)=Z$. However, since $-K_Z$ is big, we get a contradiction.
\end{proof}

If $\varphi_R$ is of divisorial type, then we call it a \emph{redundant divisorial contraction} of $Z$ (associated to $R$).
Suppose that $\varphi_R$ is of small type.
Since it is a usual $K_Z$-negative extremal small contraction,
the (klt) flip $\chi_R:Z\dashrightarrow Z^+/Z'$ exists by \cite{BCHM}:
$$
\xymatrix{
Z\ar[rd]_{\varphi_R}\ar@{-->}[rr]^{\chi_R}&&Z^+\ar[ld]^{\varphi_R^+}\\
&Z'.&
}
$$
We call the flip $\chi_R$ a \emph{redundant flip} of $Z$ (associated to $R$). We call the sequence of the redundant divisorial contractions and redundant flips of $Z$ the \emph{redundant MMP} on $Z$. If the redundant MMP on $Z$ terminates, then the resulting model is a non-redundant model of $Z$ and we denote it by $Z_{nrd}$.

Note that there do not exist infinitely many redundant divisorial contractions in the redundant MMP
by the same reason as in the usual MMP (i.e., each divisorial contraction drops the Picard number $\rho(X)$ by one).
Note that the redundant flips exist by \cite{BCHM}.
However, we have the following

\begin{conjecture}\label{conj-redflip term}
There do not exist infinite sequences of redundant flips.
\end{conjecture}

\noindent Clearly, the termination of the usual klt flips implies Conjecture \ref{conj-redflip term}.
Nonetheless, we will see that there exists a finite sequence of redundant divisorial contractions and redundant flips to a non-redundant model $Z_{nrd}$ when $-K_Z$ admits the good Zariski decomposition
(Corollary \ref{cor-nonred=relmin}).

\subsection{Redundant MMP with Zariski decomposition}\label{redmmpwithzd}

In this subsection, we study how the redundant MMP affects the geometry of the variety when the anticanonical divisor admits the good Zariski decomposition. First, we obtain the good Zariski decomposition via the anticanonical model.


\begin{proposition}\label{prop-redundant}
Let $Z$ be a $\Q$-factorial terminal variety such that $-K_Z$ is big and $-K_Z$ admits the good Zariski decomposition $-K_Z=P+N$, and let $R$ be a $K_Z$-negative extremal ray of $\overline{\NE}(Z)$ inducing the contraction
$\varphi:Z\to Z'$. Then the following are equivalent:
\begin{enumerate}
\item $R$ is a redundant extremal ray, i.e., $\Exc(\varphi) \subseteq \B_+(-K_Z)$.

\item $R \cdot P=0$.

\item  ($\varphi$ is divisorial) $-K_{Z'}$ admits the good Zariski decomposition $-K_{Z'}=P'+N'$ such that $P=\varphi^*P'$.

\item[] ($\varphi$ is small) If $\chi:Z\dashrightarrow Z^+/Z'$ is the flip, then $-K_{Z^+}$ admits the good Zariski decomposition $-K_{Z^+}=P'+N'$ such that the pull backs of $P$ and $P'$ coincide on a common resolution of $Z$ and $Z^+$.
\end{enumerate}
\end{proposition}

\begin{proof}
(1)$\Rightarrow$(2): Since it is well known that $\B_+(-K_Z)=\B_+(P)$, we have $\Exc(\varphi)\subseteq\B_+(P)$. Note also that $\Exc(\varphi)$ is covered by curves $C$ such that
$C\cdot P=0$.
Since such curves also span the ray $R$, we have $R\cdot P=0$.

(2)$\Rightarrow$(1): $R\cdot P=0$ implies that the curves $C$ contracted by $\varphi$ are contained in $\B_+(P)=\B_+(-K_Z)$.
Since such curves are movable in $\Exc(\varphi)$, we obtain $\Exc(\varphi)\subseteq\B_+(K_Z)$, hence $R$ is a redundant extremal ray.

(2)$\Rightarrow$(3): Consider the anticanonical morphism $f:=\varphi_{|mP|} : Z \rightarrow Y$ for a sufficiently divisible integer $m>0$. By Lemma \ref{zarviaanti}, we have $P=f^*(-K_Y)$.

Let $\varphi$ be a divisorial contraction.
Since $f$ contracts all the curves that are contracted by $\varphi:Z\to Z'$, it follows that $f$ factors through $Z'$. We have the following commutative diagram
\[
\xymatrix{
Z \ar[rr]^{\varphi} \ar[rd]_{f} && Z' \ar[ld]^{f'}\\
&Y.&
}
\]
We can write $-K_{Z'} = f'^{*}(-K_Y) + N'$ for some $f'$-exceptional divisor $N'$.
Since $\varphi$ is a $K_Z$-negative extremal contraction, we have
$$
-K_Z = \varphi^{*}(-K_{Z'}) - aE = f^{*}(-K_Y) + (\varphi^{*}N' - aE)
$$
for some rational number $a>0$.
Recall that $-K_Z = f^{*}(-K_Y) + N$ is the good Zariski decomposition, and hence, $N = \varphi^{*}N'-aE$.
By the negativity lemma, $N'$ is an effective $\Q$-divisor.
Since $N'$ is effective and $f'$-exceptional, we have
$$
H^0(\mathcal{O}_{Z'}(-K_{Z'})) = H^0(\mathcal{O}_{Z'}(-mf'^*(-K_Y)))=H^0(\mathcal{O}_Y(-mK_Y))
$$
for a sufficiently divisible integer $m>0$.
Thus, $-K_{Z'} = f'^{*}(-K_Y) + N'$ is the good Zariski decomposition.
Since $P=f^*(-K_Y)=\varphi^*f'^*(-K_Y)=\varphi^*P'$, the assertion for divisorial case follows.

Now let $\varphi$ be a small contraction, and $\chi:Z\dashrightarrow Z^+/Z'$ be its redundant flip:
$$
\xymatrix{
Z \ar@{-->}[rr]^{\chi}  \ar[rd]^{\varphi} \ar[rdd]_{f}&& Z^+ \ar[ldd]^{f'}  \ar[ld]_{\varphi'} \\
&Z' \ar[d] &\\
&Y.&
}
$$
Since $\chi$ is small, we have $H^0(\mathcal{O}_Z(-mK_Z))=H^0(\mathcal{O}_{Z^+}(-mK_{Z^+}))$ for a sufficiently large integer $m>0$.
By Lemma \ref{zarviaanti},  we have the good Zariski decomposition $-K_{Z^+}=f'^*(-K_Y)+N'$ where $N'$ is an effective $f'$-exceptional divisor.
Let $g:W\to Z$ and $g':W\to Z^+$ be common resolutions of $Z$ and $Z^+$.
Then since $g^*f^*=g'^*f'^*$, we have $g^*P=g'^*P'$. Thus we have shown the implication (2)$\Rightarrow$(3) for the small case.

(3)$\Rightarrow$(2): Let $C$ be a curve contracted by $\varphi$. Since $f(C)$ is a point, it follows that 
$P\cdot C=f^{*}(-K_Y)\cdot C=0$.
\end{proof}

\begin{remark}
In the proof of Proposition \ref{prop-redundant}, the bigness of $-K_Z$ is essential. In fact, the equivalence (2)$\Leftrightarrow$(3) in Proposition \ref{prop-redundant} is no longer true when $-K_Z$ is not big. To see this, let $X$ be an extremal rational elliptic surface $X_{22}$ in \cite[Theorem 4.1]{MP}. Pick an intersection point $x$ of a section of the elliptic fibration on $X$ and a singular fiber with the multiplicity one. Let $\varphi : \widetilde{X} \rightarrow X$ be the blow-up at $x$ with the exceptional divisor $E$. Then the Zariski decomposition is $-K_{\widetilde{X}}=P + N$ with $P=0$. Thus $E$ spans an extremal ray $R$ such that $R \cdot P=0$ as in (2) of Proposition \ref{prop-redundant}, but $\varphi$ does not satisfy (3) of Proposition \ref{prop-redundant} since $\kappa(-K_X)=1$.
\end{remark}

\begin{corollary}\label{cor-same anticn ring}
Let $Z$ be a $\Q$-factorial terminal variety such that $-K_Z$ is big and $-K_Z$ admits the good Zariski decomposition $-K_Z=P+N$.
Let $\varphi_R:Z\to Z'$ be the redundant contraction associated to some redundant extremal ray $R$ of $Z$.
\begin{enumerate}
 \item (divisorial case) If $\varphi_R$ is a redundant divisorial contraction, then \\$\Proj R(-K_Z)\simeq  \Proj R(-K_{Z'})$.
 \item (small case) If $\varphi_R$ is small and $\chi:Z\dashrightarrow Z^+/Z'$ is its redundant flip, then
$\Proj R(-K_Z)\simeq \Proj R(-K_{Z^+})$.
\end{enumerate}
In particular, redundant divisorial contractions and redundant flips preserve the anticanonical model.
\end{corollary}

\begin{proof}
It immediately follows from (3) of Proposition \ref{prop-redundant}.
\end{proof}


It turns out that we can obtain a non-redundant model by running the relative MMP
over the anticanonical model, which can be considered as the redundant MMP.

\begin{theorem}\label{rel MMP=red MMP}
Let $Z$ be a $\Q$-factorial terminal variety such that $-K_Z$ is big and $-K_Z$ admits the good Zariski decomposition
$-K_Z=P+N$.
Let $f:Z\to Y=\Proj R(-K_Z)$ be the anticanonical morphism to the anticanonical model $Y$.
Then any relative minimal model $Z_{min}$ of $Z$ over $Y$ is a non-redundant model $Z_{nrd}$.
\end{theorem}
\begin{proof}
Let $Z_{min}$ be a relative minimal model of $Z$ over $Y$ with a morphism $\phi :Z_{min}\to Y$.
Let $R$ be a $K_{Z_{min}}$-negative extremal ray, and let $\varphi_R:Z_{min}\to Z'$ be the associated contraction.
Since $K_{Z_{min}}$ is $\phi$-nef, the curves $C$ spanning $R$ are not contracted by $\phi$.
Thus $\Exc(\varphi_R)\not\subseteq\B_+(-K_{Z_{min}})=\Exc(\phi)$ so that $R$ is not a redundant extremal ray.
\end{proof}

\begin{corollary}\label{cor-nonred=relmin}
Let $Z$ be a $\Q$-factorial terminal variety such that $-K_Z$ is big and $-K_Z$ admits the good Zariski decomposition
$-K_Z=P+N$.
Then there exists a finite sequence of redundant divisorial contractions and redundant flips:
$$
Z=Z_0\dashrightarrow Z_1\dashrightarrow Z_2\dashrightarrow\cdots\dashrightarrow Z_N=Z_{nrd}
$$
where $Z_{nrd}$ is a non-redundant model of $Z$.
\end{corollary}
\begin{proof}
Let $f:Z\to Y=\Proj R(-K_Z)$ be the anticanonical morphism.
By \cite{BCHM}, there exists a finite sequence of divisorial contractions and flips over $Y$ terminating with a relative minimal model $Z_{min}$ over $Y$.
By Theorem \ref{rel MMP=red MMP}, $Z_{min}$ is also a non-redundant morel $Z_{nrd}$ of $Z$.

Now it remains to show that each divisorial contraction and flip over $Y$ is redundant.
Let $R\subseteq\overline{\NE}(Z_i /Y)$ be a $K_{Z_i}$-negative extremal ray with the associated birational contraction $\varphi_R:Z_i\to Z_{i+1}$ over $Y$ for $0 \leq i \leq N-1$.
It is enough to show that $R$ is a redundant extremal ray of $Z_i$.
Since $R$ is contracted over $Y$, it follows that $\Exc(\varphi_R) \subseteq \Exc(f)=\B_+(P)=\B_+(-K_X)$. Thus the extremal ray $R$ is redundant.
\end{proof}

\begin{remark}[Minimality of non-redundant model $Z_{nrd}$]\label{remk-minimality}
Let $Z,Z_i,Y,Z_{nrd}$ be as in Corollary \ref{cor-nonred=relmin}. For every intermediate variety $Z_i$ in the redundant MMP, $-K_{Z_i}$ is big and $-K_{Z_i}$ admits the good Zariski decomposition.
A non-redundant model $Z_{nrd}$ satisfies the following minimal property: if we run $K_{Z_{nrd}}$-MMP further, then it breaks some expected properties. More precisely, if $\varphi:Z_{nrd}\to Z'$ is a $K_{Z_{nrd}}$-negative extremal contraction, then we have the following:
\begin{enumerate}
 \item If $\varphi$ is divisorial or of fiber type, then it changes the anticanonical model.
 \item If $\varphi$ is small and $\chi : Z \dashrightarrow Z^+/Z'$ is the flip, then $-K_{Z^+}$ does not admit the Zariski decomposition.
\end{enumerate}
\end{remark}

\section{Structure of anticanonical maps}\label{antimapsec}

In this section, by studying the structure of anticanonical maps using the redundant MMP developed in Section \ref{redmmpsec}, we prove Theorem \ref{antimap}.

\subsection{Anticanonical maps of Fano type varieties (Proof of Theorem \ref{antimap})}
In this subsection, we prove Theorem \ref{antimap}.
We also show that there are only finitely many $\Q$-factorial Fano type varieties having the same anticanonical model (Corollary \ref{finite}).

Let $X$ be a $\Q$-Gorenstein Fano type variety, and let $Y:=\Proj R(-K_X)$ be its anticanonical model.
Then for some common resolution $f:W\to X$ and $g:W\to Y$, we have the good Zariski decomposition $f^*(-K_X)=g^*(-K_Y)+N$ and $Y=\Proj R(g^*(-K_Y))$. However, $Y$ is not the anticanonical model of $W$ in general. Theorem \ref{antimap} aims to decompose the birational map  $X\dashrightarrow Y=\Proj R(-K_X)$ into the maps strictly
between the varieties having $Y$ as the anticanonical model.

Consider a minimal terminal resolution $\varphi : X_{mint} \to X$ and a terminalization $t : Z \to Y$ of $(Y,0)$.
We start with some easy lemmas.

\begin{lemma}\label{antipres}
The varieties $X_{mint}$ and $Z$ have the same anticanonical model $Y$.
\end{lemma}

\begin{proof}
It is enough to show that $H^0(\mathcal{O}_{X_{mint}}(-mK_{X_{mint}})) = H^0(\mathcal{O}_X(-mK_X))$ and
$H^0(\mathcal{O}_{Z}(-mK_{Z})) = H^0(\mathcal{O}_X(-mK_X))$ for any sufficiently divisible integer $m>0$.
By applying the negativity lemma, we obtain $-K_{X_{mint}} = \pi^{*}(-K_X) + E$ for some effective $\pi$-exceptional divisor $E$.
This shows the first equality.

If $t : Z \to Y$ is a terminalization of $(Y,0)$, then $K_Z+E=t^*K_Y$ for some effective $t$-exceptional divisor $E$.
This shows the second equality.
\end{proof}

\begin{lemma}[{\cite[Proposition 2.10]{O}}]\label{zarsqmfano}
Let $q : X_q \to X$ be a $\Q$-factorialization.
Then there is a small birational map $X_q \dashrightarrow X'$ to
a $\Q$-factorial variety such that $-K_{X'}$ admits the good Zariski decomposition.
\end{lemma}

\begin{proof}
We always have the divisorial Zariski decomposition $-K_{X_q} = P+N$ with $P$ movable (see \cite[Chapter III]{Na}). Since $X_q$ is a Mori dream space, there is a small birational map $s : X_q \dashrightarrow X'$ to a $\Q$-factorial variety such that $P':=s_*P$ is nef. Set $N':=s_*N$. Then $-K_{X'} = P'+N'$ is the good Zariski decomposition.
\end{proof}

\begin{lemma}\label{flopsft}
Every small $\Q$-factorial modification $X_q \dashrightarrow X'$ (e.g., as in Lemma \ref{zarsqmfano})
can be decomposed into a finite sequence of log flops of klt Calabi-Yau pairs.
\end{lemma}
\begin{proof}
By Lemma \ref{cy}, $X_q$ and $X'$ are $\Q$-factorial varieties of klt Calabi-Yau type. Thus they are connected by log flops (cf. \cite[Corollary 1.1.3]{BCHM}).
\end{proof}

Now we prove Theorem \ref{antimap}.

\begin{proof}[Proof of Theorem \ref{antimap}]
Recall that $q : X_q \to X$ is a $\Q$-factorialization of $X$.
First we take a small birational map $s : X_q \dashrightarrow X'$ to another $\Q$-factorial variety as in Lemma \ref{zarsqmfano}
which is a finite sequence of log flops by Lemma \ref{flopsft}.
Then take a minimal terminal resolution $\varphi: X'_{mint}\to X'$ of $X'$ (Proposition \ref{minterres}).
 By Lemma \ref{zarviaanti}, we have the anticanonical morphism $f' : X' \to Y$, and hence, we also have the anticanonical morphism $\widetilde{f}: X'_{mint} \to Y$. Then there exists a birational map $X'_{mint}\dashrightarrow X'_{nrd}$ (which can be considered as the redundant MMP) by Corollary \ref{cor-nonred=relmin}
with a birational morphism $\pi: X'_{nrd}\to Y$.
\end{proof}

The minimal property of a minimal terminalization is explained in Remark \ref{remk-minimality}.
Terminalizations of $(Y,0)$ are maximal in the sense of Proposition \ref{maxismax}.
We say that a variety $Y$ is \emph{birationally dominated} by a variety $Z$ if
there is a birational map $Z\dashrightarrow Y$ which does not extract any divisors.

\begin{proposition}\label{maxismax}
Any $\Q$-Gorenstein Fano type variety $X$ whose anticanonical model is $Y$ is birationally dominated by some terminalization of $(Y,0)$.
More precisely, there exist a small birational map $X\dashrightarrow X'$ to a $\Q$-factorial variety and a terminalization $Z\to Y$ of $(Y,0)$
such that we have a factorization into morphisms $Z\to X' \to Y$.
If there is a morphism $X\to Y$, then we can let $X=X'$.
\end{proposition}

\begin{proof}
Since a $\Q$-factorialization is small birational, by using Lemma \ref{zarsqmfano}, we can take a small birational map $X\dashrightarrow X'$ to a $\Q$-factorial variety $X'$ such that $-K_{X'}$ admits the good Zariski decomposition $-K_{X'}=P+N$. We obtain a birational morphism $f:X'\to Y$ by applying Lemma \ref{zarviaanti}.
By \cite[Lemma 1.6]{HK}, for each $f$-exceptional prime divisor $E$ over $X'$, we have
$a(E;Y,0) \leq 0$.
By \cite[Theorem 1.4.3]{BCHM}, we can take a terminalization $g:Z\to X'$
of $(X',N)$. Since $a(E';Y,0)=a(E';X',N)$ for any $g$-exceptional divisor $E'$, it follows that the composition $f\circ g:Z\to Y$ is a terminalization of $(Y,0)$.
\end{proof}


\begin{corollary}\label{finite}
For a given klt Fano variety $Y$, there are only finitely many $\Q$-factorial Fano type varieties having $Y$ as the anticanonical model.
\end{corollary}

\begin{proof}
Recall that any terminalization $Z$ of $(Y,0)$ has the anticanonical model $Y$ by Lemma \ref{antipres}, and thus by Theorem \ref{antimodel}, $Z$ is a variety of Fano type. By Theorem \ref{antimap} and Proposition \ref{maxismax}, $Z$ birationally dominates any $\Q$-factorial Fano type variety having $Y$ as the anticanonical model. Since any $\Q$-factorial variety of Fano type is a Mori dream space, there are only finitely many birational contraction maps $Z \dashrightarrow X'$ to a $\Q$-factorial Mori dream space. Thus the assertion follows.
\end{proof}

Corollary \ref{finite} in particular implies that there are only finitely many terminalizations of $Y$.

\subsection{Anticanonical maps of Calabi-Yau type varieties}\label{antimapcysec}
In this subsection, we state and prove an analogue of Theorem \ref{antimap} allowing lc singularities.
We first prove that if the variety $Y$ in Corollary \ref{finite} is not klt, then the situation changes.

\begin{proposition}\label{no maxt}
Let $Y$ be a $\Q$-Gorenstein variety with ample anticanonical divisor. If $Y$ contains  singularities worse than klt, then there are infinitely many $\Q$-factorial projective varieties $Z$ having terminal singularities with the anticanonical model $Y$.
\end{proposition}

\begin{proof}
By \cite[Corollary 1.4.4]{BCHM}, we can take a birational morphism $\pi : X \rightarrow Y$
from a $\Q$-factorial variety $X$ such that for some effective $\Q$-divisors $\Gamma_1$ and $\Gamma_2$ on $X$,
we have
$$
-K_X = \pi^{*}(-K_Y) + \Gamma_1+\Gamma_2
$$
where $(X,\Gamma_1)$ is klt and the coefficients of $\Gamma_2$ are at least $1$.
Note that the anticanonical model of $X$ is $Y$. Now let $\psi : X_{mint} \rightarrow X$ be a minimal terminal resolution.
Then the anticanonical model of $X_{mint}$ is also $Y$. We may write
$$
-K_{X_{mint}} = \psi^{*}(-K_X)+E = \psi^{*}\pi^{*}(-K_Y) + \psi^{*}\Gamma_1 + \psi^{*}\Gamma_2 + E,
$$
where $E$ is an effective $\psi$-exceptional $\Q$-divisor. Now let $\phi : Z \rightarrow X_{mint}$ be a blow-up at codimension two locus $V$ such that $V$ is contained in an irreducible component of $\psi^{-1}_{*}\Gamma_2$ and $V$ contains smooth points of $X_{mint}$.
Then we have
$$
-K_{Z} = \phi^{*}(-K_{X_{mint}}) - F =  \phi^{*}\psi^{*}\pi^{*}(-K_Y) + \phi^{*}\psi^{*}\Gamma_1 +
(\phi^{*}\psi^{*}\Gamma_2 -F)+ \phi^{*}E,
$$
where $F$ is a $\phi$-exceptional prime divisor. By the conditions on $V$, the $\Q$-divisor $\phi^{*}t^{*}\Gamma_2 -F$ is effective, and hence, the anticanonical model of $Z$ is $Y$.
By blowing up a codimension two locus $V'$ contained in an irreducible component of $(\psi\circ\phi)^{-1}_*\Gamma_2$ containing smooth points of $Z$,
we obtain a variety $Z'$ whose anticanonical model is $Y$.
By successively blowing up codimension two locus similarly, we can obtain infinitely many varieties having $Y$ as the anticanonical model.
\end{proof}

Proposition \ref{no maxt} shows that there are no `maximal' model in the sense of Proposition \ref{maxismax}.
Despite this inconvenience, we still have the following theorem.

\begin{theorem}\label{antimapofcy}
Let $X$ be a $\Q$-Gorenstein variety of Calabi-Yau type such that $-K_X$ is big and $R(-K_X)$ is finitely generated, and let $f: X \dashrightarrow Y=\Proj R(-K_X)$ be the anticanonical map. Then we can take $\Q$-factorial varieties $X_d$, $X'$, $X'_{mint}$, and $X'_{nrd}$ having $Y$ as the anticanonical model with the following commutative diagram
$$
\xymatrix{
& X'_{mint} \ar@{-->}[r]^{r}  \ar[d]_{\varphi} \ar[rd]^{\widetilde{f}} & X'_{nrd} \ar[d]^{\pi}\\
X_d \ar@{-->}[r]^s \ar[d]_d   & X' \ar[r]^{f'} & Y \\
X \ar@{-->}@/_1pc/[rru]^f &  &
}
$$
where $d$ is a dlt blow-up of an lc Calabi-Yau pair $(X, \Delta)$, $s$ is a finite sequence of log flops, $f'$ and $\widetilde{f}$ are anticanonical morphisms, $\varphi$ and $\pi$ are minimal terminal resolutions, and $r$ is a redundant MMP to a non-redundant model $X'_{nrd}$.
\end{theorem}

\begin{remark}
The varieties $X'$, $X'_{mint}$, and $X_{nrd}'$ are of lc weak Fano type by Lemmas \ref{zarviaanti} and \ref{cy=lc}.
\end{remark}

To prove Theorem \ref{antimapofcy}, we need some lemmas.

\begin{lemma}
Let $X$ be a $\Q$-Gorenstein variety of Calabi-Yau type such that $-K_X$ is big and $R(-K_X)$ is finitely generated, and let $\varphi: X_{mint} \to X$ be a minimal terminal resolution. 
Then $X_{mint}$ is a variety of Calabi-Yau type having $Y=\Proj R(-K_X)$ as the anticanonical model.
\end{lemma}

\begin{proof}
As in the proof of Lemma \ref{antipres}, the assertion follows from the negativity lemma.
\end{proof}

The following is the key lemma in the proof of Theorem \ref{antimapofcy} (cf. Lemma \ref{zarsqmfano}).

\begin{lemma}\label{zarofcy}
Let $X$ be a $\Q$-Gorenstein variety of Calabi-Yau type such that $-K_X$ is big and $R(-K_X)$ is finitely generated, and
let $d : X_d \to X$ be a dlt blow-up of an lc Calabi-Yau pair $(X, \Delta)$.
Then we have the following:
\begin{enumerate}
 \item $X_d$ is a variety of Calabi-Yau type having $Y=\Proj R(-K_X)$ as the anticanonical model.
 \item There is a small birational map $X_d \dashrightarrow X'$ to a $\Q$-factorial variety such that $-K_{X'}$ admits the good Zariski decomposition.
\end{enumerate}
\end{lemma}

\begin{proof}
For the given dlt blow-up $d : X_d \to X$, there exists
an effective $\Q$-divisor $\Delta_d$ on $X_d$ such that $(X_d, \Delta_d)$ is a dlt pair and $K_{X_d}+\Delta_d = d^*(K_X + \Delta)$.
For some log resolution $g : W \rightarrow X_d$ of $(X_d,\Delta_d)$,
we may write
$$
-K_W=g^*(-(K_{X_d}+\Delta_d ))+ g_*^{-1} \Delta_d + G
$$
where $G$ is a $g$-exceptional divisor. Since $(X_d, \Delta_d)$ is dlt, every coefficient of $G$ is less than 1.
We may consider $h : W \rightarrow Y$ also as a log resolution of $Y$

Let $\mathcal{E}$ be the set of all prime $h$-exceptional divisors on $W$ that are not $g$-exceptional.
Then for any divisor $E\in \mc E$, we have $-1\leq a(E;Y,0)\leq 0$.
Indeed, since $K_X+\Delta\sim_\Q0$, we can write
$$
K_W=g^*(K_{X_d}+\Delta_d )-g_*^{-1} \Delta_d - G\sim_\Q  h^*(K_Y) + D
$$
where $D=-g_*^{-1}\Delta_d-G+h^*(-K_Y)$.
Thus for any $E\in\mc E$, we have $a(E;Y,0)=\mult_ED$ and
$-1\leq \mult_ED=\mult_E(-g_*^{-1} \Delta_d) \leq 0$ holds since  $h^*(-K_Y)$ is semiample.
Furthermore, note that $-1=a(E,Y,0)$ implies that $E$ is a component of $-g_*^{-1} \Delta_d$.
In particular, such $E$ is not $g$-exceptional and $E\in\mc E$.
Therefore, we can apply \cite[Theorem 4.1]{F2} for $\mc E$ and there is a birational morphism $b \colon X' \rightarrow Y$
such that $X'$ is $\Q$-factorial and the $b$-exceptional divisors are exactly the members of $\mathcal{E}$.
Thus there is a small birational map $s : X_d \dashrightarrow X'$.
Then by Lemma \ref{zarviaanti}, $-K_{X'}$ admits the good Zariski decomposition. We have shown the assertion (2). Now we may write
$$
-K_{X'} = b^*(-K_Y) - \sum_{E \in \mathcal{E}} a(E;Y,0) E.
$$
Since $ - \sum_{E \in \mathcal{E}} a(E;Y,0) E$ is an effective $b$-exceptional divisor, it follows that the anticanonical model of $X'$ is $Y$.
Thus we get the assertion (1).
\end{proof}

\begin{lemma}\label{flopscy}
Let $X$ be a $\Q$-Gorenstein variety of Calabi-Yau type such that $-K_X$ is big and $R(-K_X)$ is finitely generated, and  $d : X_d \to X$ a dlt blow-up of an lc Calabi-Yau pair $(X, \Delta)$.
Then every small birational map $s: X_d \dashrightarrow X'$ to a $\Q$-factorial variety such that $-K_{X'}$ admits the good Zariski decomposition can be decomposed into a finite sequence of log flops of lc Calabi-Yau pairs.
\end{lemma}

\begin{proof}
There is an effective $\Q$-divisor $\Delta_d$ on $X_d$ such that $(X_d, \Delta_d)$ is a dlt pair and $K_{X_d}+\Delta_d = d^*(K_X + \Delta) \sim_{\Q} 0$.
We can take ample divisors $A$ on $X_d$ and $A'$ on $X'$ such that $s_*A + A'$ is ample on $X'$ and $(X_d, \Delta_d+ A+s^{-1}_* A')$ is a dlt pair.  Since $K_{X'}+s_*(\Delta_d) + s_*A + A' \sim_{\Q} s_*A + A'$ is ample, $(K_{X_d}+\Delta_d + A+s^{-1}_* A')$-MMP with scaling terminates with a log minimal model $(X', s_*(\Delta_d) + s_*A + A')$ by \cite[Theorem 1.9]{B}. Thus the assertion follows.
\end{proof}

Now we can prove Theorem \ref{antimapofcy} similarly as in the proof of Theorem \ref{antimap}. We leave the details to the interested readers.

\section{Comparison between log Calabi-Yau pairs and log weak Fano pairs}\label{cywfsec}

In \cite[Theorem 1.1]{HP2}, the authors proved that a surface with big anticanonical divisor whose anticanonical model is an lc Fano variety is of lc weak Fano type.
It is natural to ask whether the same holds true in higher dimensions.
However, Theorem \ref{lcsing} shows that such a variety is only of Calabi-Yau type.
The aim in this section is to compare the following two classes:
$$
CY_n:=\left\{ X \left|
\begin{array}{l}\text{$X$ is a $\Q$-Gorenstein variety of Calabi-Yau type such that}\\
 \text{$\dim X=n$, $-K_X$ is big and $R(-K_X)$ is finitely generated.}
 \end{array}\right.\right\}_, \text{ and}
$$
$$
WF_n:=\left\{ X \left|
\begin{array}{l}\text{$(X, \Delta)$ is an lc weak Fano pair for some $\Q$-divisor $\Delta$ such that}\\
 \text{$\dim X=n$, $-K_X$ is $\Q$-Cartier and $R(-K_{X})$ is finitely generated.}
 \end{array}\right.\right\}_.
$$

We can easily see that $CY_1=WF_1=\{ \P^1 \}$.
By \cite[Theorem 1.1]{HP2} and Theorem \ref{lcsing}, we also have $CY_2=WF_2$.
For $n>2$, Corollary \ref{wf=>cy} shows that $WF_n \subseteq CY_n$ and
the inclusion is strict in general by Example \ref{ex-lcCY neq lcwFano}.

\begin{corollaryint}\label{wf=>cy}
Let $(X, \Delta)$ be an lc weak Fano pair. If $-K_X$ is $\Q$-Cartier and $R(-K_X)$ is finitely generated, then $X$ is a variety of Calabi-Yau type.
In particular, $WF_n\subseteq CY_n$.
\end{corollaryint}

\begin{proof}
By Theorem \ref{lcsing}, we only have to show that the anticanonical model $Y=\Proj R(-K_X)$ contains lc singularities. Let $f : W \rightarrow X$ be a log resolution of $(X, \Delta)$ such that $f^{*}(-K_X)$ admits the good Zariski decomposition $f^{*}(-K_X)=P+N$. Then we obtain the birational morphism $g : W \rightarrow Y:=\Proj R(-K_X)$, and we have $P=g^{*}(-K_Y)$. By taking further blow-ups, we may assume that $g : W \rightarrow Y$ is a log resolution of $(Y, 0)$. We may write
$$
-K_W - f^{-1}_*\Delta= f^{*}(-K_X -\Delta) + F = g^*(-K_Y) + N - f^{*}\Delta + F,
$$
where $F$ is a $f$-exceptional divisor. Note that $f^{*}(-K_X)=P+N$ is the Fujita-Zariski decomposition and $f^{*}(-K_X)-f^{*}\Delta$ is nef.
Thus we have  $f^{*}(-K_X)-f^{*}\Delta\leq P=g^*(-K_Y)$, and hence,  $-(N-f^{*}\Delta)$ is an effective $\Q$-divisor.
Since every coefficient of $F$ is at most 1, it follows that every coefficient of  $N - f^{*}\Delta + F + f^{-1}_*\Delta$ is also at most 1. Thus $Y$ contains lc singularities.
\end{proof}

If $R(-K_X)$ is not finitely generated, then Corollary \ref{wf=>cy} is no longer true by the following example.

\begin{exampleint}\label{ex-lcwFano neq lcCY}(cf. \cite[Example 5.5]{G})
Let $E$ be a non-split vector bundle of degree 0 and rank 2 on an elliptic curve $C$, and let $S:=\P(E)$ be a ruled surface. Consider the smooth projective variety $X:=\P(\mathcal{O}_S \oplus \mathcal{O}_S(-A))$, where $A$ is a very ample divisor on $S$ such that the natural map $\Sym^rH^0(\mathcal{O}_S(A)) \to H^0(\mathcal{O}_S(rA))$ is surjective for every integer $r>0$. Denote by $E$ the tautological divisor of $\mathcal{O}_S \oplus \mathcal{O}_S(-A)$. Then $(X,E)$ is an lc weak Fano pair (see \cite[Basic construction 5.1]{G}). We can easily show that $H^0(\mathcal{O}_X(-mK_X-mE))=H^0(\mathcal{O}_X(-mK_X))$ for every integer $m>0$. Thus if $(X, \Delta)$ is a Calabi-Yau pair for some $\Q$-divisor $\Delta$, then $\Delta \geq E$. Since $(X,E)$ does not admit $\Q$-complements (see \cite[Example 5.5]{G}), it follows that $X$ is not a variety of Calabi-Yau type. In particular, $R(-K_X)$ is not finitely generated by Corollary \ref{wf=>cy}.
\end{exampleint}

Corollary \ref{cor-CY is closed} shows that $CY_n$ is invariant under small birational maps whereas
the following example shows that $WF_n$ is not.
In particular, the inclusion $WF_n \subseteq CY_n$  is strict.

\begin{exampleint}\label{ex-lcCY neq lcwFano}
Let $C$ be an elliptic curve, and let $A$ be a very ample divisor on $C$. Consider the smooth projective variety $X:=\P(\mathcal{O}_C(A) \oplus \mathcal{O}_C(A) \oplus \mathcal{O}_C(-2A))$ with the fibration $\pi : X \to C$. Then we have the following.
\newline

\noindent\textbf{Claim.}
(1) $X$ is of Calabi-Yau type such that $-K_X$ is big and movable.\\
(2) The $\C$-algebra $\bigoplus_{m,n \geq 0} H^0(\mathcal{O}_X(-mK_X + n\pi^*A))$ is finitely generated. In particular, $R(-K_X)$ is finitely generated.

\begin{proof}[Proof of Claim]
By \cite[Proposition 2.10]{CG}, $X$ is of Calabi-Yau type. Since $2A + A + (-2A)=A$ is big, the tautological line bundle $\mathcal{O}_X(1)$ is big so that $\mathcal{O}_X(-K_X)=\mathcal{O}_X(3)$ is also big. To show that $-K_X$ is movable, we regard $X$ as a locally trivial $\P^2$-bundle over $C$ so that the global section of $H^0(\mathcal{O}_X(1))$ is locally given by $s_1 x_1 + s_2 x_2$, where $x_0, x_1, x_2$ are coordinates for $\P^2$ and $s_1, s_2$ are global sections of $H^0(\mathcal{O}_C(A))$. Since the base locus of $|\mathcal{O}_X(1)|$ is locally $V(x_0)$, it follows that $-K_X$ is movable. Finally, we can easily check that $\bigoplus_{m,n \geq 0} H^0(\mathcal{O}_X(-mK_X + n\pi^*A))$ is a splitting subring of $\left( \bigoplus_{n \geq 0} H^0(\mathcal{O}_C(nA)) \right) [x_0, x_1, x_2]$, and thus it is finitely generated.
\end{proof}

Now suppose that $(X, \Delta)$ is an lc weak Fano pair for some effective $\Q$-divisor $\Delta$. Since $H^1(\mathcal{O}_X) \neq 0$, by Kodaira vanishing (\cite{A}, \cite[Theorem 2.42]{F1}), $-(K_X + \Delta)$ is not ample. If the morphism $\varphi$ induced by $|-m(K_X + \Delta)|$ for a sufficiently large integer $m>0$ is a divisorial contraction, then it is the anticanonical morphism. However, $-K_X$ does not admit the Zariski decomposition on $X$, which is a contradiction to Lemma \ref{zarviaanti}. Thus $\varphi$ is small, and by the claim (2), there is $-(K_X + \Delta)$-flop $f : X \dashrightarrow X'$. Note that $(X', \Delta':=f_*\Delta)$ is an lc weak Fano pair. Then $(X', (1-\epsilon)\Delta')$ is an lc Fano pair for a sufficiently small rational number $\epsilon >0$. By Kodaira vanishing (\cite[Theorem 2.42]{F1}), $H^1(\mathcal{O}_{X'}) =0$, which is a contradiction. Thus there is no effective $\Q$-divisor $\Delta$ on $X$ such that $(X, \Delta)$ is an lc weak Fano pair.
\end{exampleint}

Although we have $WF_n\subsetneq CY_n$ in general, we see that the difference is actually `\emph{small}' by Corollary \ref{cy=>weakfano}.

\begin{lemmaint}\label{cy=lc}
Let $X$ be a $\Q$-Gorenstein variety of Calabi-Yau type such that $-K_X$ is big and $-K_X$ admits the good Zariski decomposition $-K_X=P+N$. Then $(X, N)$ is an lc weak Fano pair.
\end{lemmaint}

\begin{proof}
Note that $P:=-(K_X+N)$ is nef and big.
There exists an effective $\Q$-divisor $\Delta$ such that $(X,\Delta)$ is an lc Calabi-Yau pair.
Considering the decompositions $-K_X=P+N\sim_\Q0+\Delta$, we have $N\leq\Delta$ by Remark \ref{remk-Zariski decomp} (2).
Thus $(X,N)$ is an lc pair.
\end{proof}

\begin{corollaryint}\label{cy=>weakfano}
Let $X$ be a $\Q$-Gorenstein variety of Calabi-Yau type such that $-K_X$ is big and $R(-K_X)$ is finitely generated, and let $d: X_d \to X$ be a dlt-blow-up of an lc Calabi-Yau pair $(X, \Delta)$.
Then there is a small birational map $X_d \dashrightarrow X'$ such that $(X', \Delta')$ is an lc weak Fano pair for some $\Q$-divisor $\Delta'$.
\end{corollaryint}

\begin{proof}
By Lemma \ref{zarofcy}, there is a small birational map $X_d \dashrightarrow X'$ such that $-K_{X'}$ admits the good Zariski decomposition $-K_{X'} = P+N$. Since $\Proj R(-K_X)= \Proj R(-K_{X'})$, by Theorem \ref{lcsing}, $X'$ is a variety of Calabi-Yau type. By Lemma \ref{cy=lc}, $(X', N)$ is an lc pair. Since $-(K_{X'}+N)=P$ is nef and big, $(X', N)$ is an lc weak Fano pair.
\end{proof}

\section*{Acknowledgements}
We would like to thank the referees for reading the manuscript carefully and providing helpful suggestions.
We also would like to thank Professor V. Shokurov for several comments.


\end{document}